\documentclass{amsart}

\usepackage{calrsfs}
\usepackage{nicefrac}
\usepackage{url}
\usepackage{hyperref}

\newtheorem{theorem}{Theorem}
\newtheorem*{theorem*}{Theorem (Lehmann--Scheff\'e)}
\newtheorem{corollary}[theorem]{Corollary}
\newtheorem{proposition}[theorem]{Proposition}

\theoremstyle{definition}

\newtheorem{remark}[theorem]{Remark}

\theoremstyle{definition} 
\newtheorem*{remark*}{Remark}
\newtheorem{example}[theorem]{Example}

\newcommand{\R}{\mathbb{R}}

\newcommand{\X}{X}

\newcommand{\A}{\Sigma}
\newcommand{\EE}{\mathcal{E}}
\newcommand{\LL}{V}
\newcommand{\PP}{\mathcal{P}}

\renewcommand{\l}{\mathfrak{L}}
\renewcommand{\ll}{\mathfrak{c}}
\newcommand{\sq}{\operatorname{\mathfrak{s}\mathfrak{q}}}
\newcommand{\C}{\mathcal{C}}

\newcommand{\al}{\alpha}
\newcommand{\be}{\beta}
\newcommand{\Ga}{\Gamma}
\newcommand{\la}{\lambda}
\newcommand{\La}{\Lambda}

\newcommand{\Si}{\Sigma}
\renewcommand{\th}{\theta}
\newcommand{\Th}{\Theta}
\newcommand{\E}{\operatorname{\mathsf{E}}}
\renewcommand{\P}{\operatorname{\mathsf{P}}}
\newcommand{\Var}{\operatorname{\mathsf{Var}}}
\newcommand{\Cov}{\operatorname{\mathsf{Cov}}}

\newcommand{\ii}{\operatorname{I}}

\begin{document}

\title[Best unbiased estimators]{A characterization of best unbiased estimators}

\author{Iosif Pinelis}
\address{Department of Mathematical Sciences, Michigan Technological University}
\email{ipinelis@mtu.edu}




\begin{abstract} 
A simple characterization of uniformly minimum variance unbiased estimators (UMVUEs) is provided (in the case when the sample space is finite) in terms of a linear independence condition on the likelihood functions corresponding to the possible samples.  
The crucial observation in the proof is that, if a UMVUE exists, then, 
after an appropriate cleaning of the parameter space, 
the nonzero likelihood functions  
are eigenvectors of an ``artificial'' matrix of Lagrange multipliers, and the values of the UMVUE are eigenvalues of that matrix. The characterization is then extended to best unbiased estimators with respect to arbitrary convex loss functions. 
\end{abstract}

\maketitle

Let $(\X,\A)$ be a measurable space, so that $\X$ is a set and $\A$ is a sigma-algebra of subsets of $\X$. The set $\X$ is to be interpreted as the set of all possible statistical samples and will be assumed nonempty. Any mapping of $\X$ to $\R$ that is measurable with respect to the sigma-algebra $\A$ over $\X$ and the Borel sigma-algebra over $\R$ is called a (real-valued) statistic or, equivalently, a (real-valued) estimator.  

Let $\PP:=(\P_\th)_{\th\in\Th}$ be a family of probability measures on $\A$, 
where $\Th$ is a nonempty 
set, called the parameter space. 
The triple $(\X,\A,\PP)$ is called a statistical model. 
Let us say that a set $N\in\A$ is a null set (for the model) if $\P_\th(N)=0$ for all $\th\in\Th$.  

For each $\th\in\Th$, let $\E_\th$ denote the expectation with respect to the probability measure $\P_\th$. For $j=1,2$, let $L^j=L^j(\X,\A,\PP)$ stand for the set of all statistics $T$ such that $\E_\th|T|^j<\infty$ for all $\th\in\Th$. 

Let $b$ be any function from $\Th$ to $\R$. A statistic $T\in L^1$ is called unbiased for the function $b$ if 
\begin{equation}\label{eq:ET=b}
	\text{$\E_\th T=b(\th)$ for all $\th\in\Th$;}
\end{equation} 
on the other hand, for a given statistic $T$, the function $b$ satisfying \eqref{eq:ET=b} may be called the expectation function of $T$. 
Let $\EE_b$ denote the set of all unbiased estimators of the function $b$. 
In particular, $\EE_0$ will denote the set of all unbiased estimators of the zero function. 

A statistic $T\in L^2$ is called a uniformly minimum variance unbiased estimator (UMVUE) of the function $b$ if (i) $T\in\EE_b$ and (ii) for any $\tilde T\in\EE_b$ 
and all $\th\in\Th$ one has $\Var_\th T\le\Var_\th\tilde T$ or, equivalently, $\E_\th T^2\le\E_\th{\tilde T}^2$. 
If $T$ is a UMVUE of some function $b$, let us say simply say that $T$ is a UMVUE. 

Let us say that a statistic $T$ is sufficient if for any statistic $S\in L^1$ there exists a statistic $S_T$ 
that is, for each $\th\in\Th$, a version of the conditional expectation $\E_\th(S|T)$; the key here is that the statistic $S_T$ is the same for all $\th\in\Th$. This definition is slightly more convenient than, and is easily seen to be equivalent to, the usual definition of a sufficient statistic; see e.g.\ \cite[top of page~311]{lehmann-scheffeI}. 

A statistic $T$ is called complete if, for any Borel-measurable function $u$ from $\R$ to $\R$ such that $u\circ T\in\EE_0$, one has $u\circ T=0$ except on a null set. 

The Lehmann--Scheff\'e theorem 
\cite[Theorem~5.1]{lehmann-scheffeI} 
is as follows. 

\begin{theorem*}
Let $T$ be a complete sufficient statistic. Let a Borel-measurable function $u$ from $\R$ to $\R$ be such that $u\circ T\in L^1$. Then $u\circ T$ is a UMVUE. 
\end{theorem*}

Throughout the rest of this paper, assume that the set $X$ of all samples is finite and $\A$ is the set of all subsets of $X$, unless specified otherwise. Thus, the set of all statistics will be the same as the set $\R^X$ of all functions from $X$ to $\R$. 

For each sample $x\in X$, let $\ell_x$ stand for the corresponding likelihood function, mapping $\Th$ to $\R$ and defined by the formula 
\begin{equation}
	\ell_x(\th):=\P_\th(\{x\})\quad\text{for all }\th\in\Th. 
\end{equation}

For each $t\in T(X):=\{T(x)\colon x\in X\}$, take any (linear) basis $B_t$ 
of the set $\{\ell_x\colon T(x)=t\}$ of likelihood functions; in particular, $B_t$ will be necessarily empty if the likelihood functions $\ell_x$ are zero for all $x\in X$ such that $T(x)=t$. (As usual, it is assumed here that the sum of an empty family is zero.) 

\begin{theorem}\label{prop:iff}
A statistic $T$ is a UMVUE iff 
the union $\bigcup_{t\in T(X)}B_t$ of the bases is linearly independent. 
\end{theorem}

Theorem~\ref{prop:iff} will be proved at the end of this note. 

Another characterization of UMVUEs was provided by Theorem~5 of Bahadur~\cite{bahadur57}, which implies that there is a sigma-algebra $\Si_0$ over $X$ such that a statistic $T$ is a UMVUE iff $T$ is $\Si_0$-measurable. The sigma-algebra $\Si_0$ can be described (see e.g.\ \cite{schmett-strasser}) as the set of all subsets of $X$ whose indicator is a UMVUE. 
It appears that the necessary and sufficient linear-independence condition given in Theorem~\ref{prop:iff} above is more explicit and easier to check than the $\Si_0$-measurability condition. On the other hand, Bahadur's characterization of UMVUEs holds not only for finite sets $X$ of all samples. 

Consider the matrix $P:=[\P_\th(\{x\})\colon\th\in\Th,x\in\X]$, so that the rows and columns of $P$ represent, respectively, the probability mass functions (say $P_\th$) of the probability measures $\P_\th$ on $X$ (for $\th\in\Th$) and 
the likelihood functions $\ell_x$ (for $x\in X$). 

\begin{example}
Suppose that $X=\{1,2,3,4\}$ 
and $\Th=\{1,2\}$. 

If $P=P_1:=\begin{bmatrix}
 \nicefrac{1}{3} & \nicefrac{1}{3} & \nicefrac{1}{3} & 0 \\
 \nicefrac{1}{6} & \nicefrac{1}{3} & \nicefrac{1}{2} & 0 \\
\end{bmatrix}$, 
then any two of the first three columns of $P$ are linearly independent, but of course no three columns of $P$ are so. It follows immediately by Theorem~\ref{prop:iff} that a statistic $T$ is a UMVUE here iff $T(1)=T(2)=T(3)$. It then follows that here the sigma-algebra $\Si_0$ is generated by the set $\{1,2,3\}$, so that 
\begin{equation}\label{eq:Si_0}
	\Si_0=\big\{\emptyset,\{1,2,3\},\{4\},\{1,2,3,4\}\big\}. 
\end{equation}

Alternatively, one can find the sigma-algebra $\Si_0$ by using its mentioned description as the set of all subsets of $X$ whose indicator is a UMVUE. It is well-known and easy to see that a statistic $T$ is a UMVUE iff for any statistic $H$ one has the implication 
\begin{equation}\label{eq:implies}
\text{$H\in\EE_0\implies TH\in\EE_0$.} 	
\end{equation}
Next, the condition $H\in\EE_0$ means that  
the function $H$ from $X=\{1,2,3,4\}$ to $\R$, identified with the row $[H(1),\dots,H(4)]$, is in the orthogonal complement, say $O$, of the row space of the matrix $P$ to the set $\R^X$ of all statistics with respect to the usual inner product, defined by the formula $T\cdot S:=T(1)S(1)+\dots+T(4)S(4)$. 
One finds that, for $P=P_1$, the orthogonal complement $O$ is the linear span of two rows, say $H_1:=[1,-2,1,0]$ and $H_2:=[1,-2,1,1]$. Thus, the indicator $\ii_A$ of a subset $A$ of $X=\{1,2,3,4\}$ is a UMVUE iff $(H_j\ii_A)\cdot P_\th=0$ for $j\in\{1,2\}$ and $\th\in\Th=\{1,2\}$. One can then check that the set of all such subsets $A$ is the same as the sigma-algebra $\Si_0$ in \eqref{eq:Si_0}.  

If $P=P_2:=\begin{bmatrix}
 \nicefrac{1}{2} & \nicefrac{1}{4} & \nicefrac{1}{4} & 0 \\
 \nicefrac23 & \nicefrac{1}{3} & 0 & 0 \\
\end{bmatrix}$, 
then the first two columns of $P$ are linearly dependent, whereas their linear span is linearly independent of the third column. 
It follows immediately by Theorem~\ref{prop:iff} that a statistic $T$ is a UMVUE here iff $T(1)=T(2)$. It then follows that here the sigma-algebra $\Si_0$ is generated by the sets $\{1,2\}$, $\{3\}$, $\{4\}$, so that 
$
	\Si_0=\big\{\emptyset,\{3\},\{4\},\{1,2\},\{3,4\},\{1,2,3\},\{1,2,4\},\{1,2,3,4\}\big\}. 
$ 
Similarly to the case $P=P_1$, one can check the latter sigma-algebra $\Si_0$ coincides with the set of all subsets of $X$ whose indicator is a UMVUE. 

At least for these two examples, with $P=P_1$ and $P=P_2$, it appears that indeed the necessary and sufficient linear-independence condition given in Theorem~\ref{prop:iff} is more explicit and easier to check than the $\Si_0$-measurability condition. However, it also appears that there is a duality between these two necessary-and-sufficient conditions, one in terms of the linear independence of some of the columns of the matrix $P$ and the other one expressible in terms of the rows of $P$. 
\end{example}

\begin{corollary}\label{cor:u(T)}
If a statistic $T$ is a UMVUE, then $u\circ T$ is so for any function $u\colon\R\to\R$. 
\end{corollary}

\begin{proof}[Proof of Corollary~\ref{cor:u(T)}] 
This follows immediately from Theorem~\ref{prop:iff} because, for any $s\in(u\circ T)(X)$, there is a basis of the set $\{\ell_x\colon(u\circ T)(x)=s\}$ contained in the union $\bigcup\{B_t\colon t\in T(X), u(t)=s\}$. 
\end{proof}

\begin{corollary}\label{cor:complete}
Suppose that a statistic $T$ is a UMVUE. Then $T$ is complete. 
\end{corollary}

\begin{proof}[Proof of Corollary~\ref{cor:complete}] 
Take any function $u\colon\R\to\R$ such that $u\circ T\in\EE_0$. Then, by Corollary~\ref{cor:u(T)}, $u\circ T$ is a UMVUE of the zero function. On the other hand, the zero statistic is clearly a UMVUE of the zero function. So, $\E_\th(u\circ T)^2=\Var_\th(u\circ T)=\Var_\th 0=0$ for all $\th$, whence $u\circ T=0$ except on a null set.  
Thus, $T$ is complete. 
\end{proof}

Another relation of the necessary-and-sufficient linear independence condition in Theorem~\ref{prop:iff} with the completeness is presented in the following corollary.  

\begin{corollary}\label{cor:UMVUE,compl}
Suppose that a statistic $T$ is such that for each $t\in T(X)$ the basis $B_t$ is a singleton set. 
Then $T$ is a UMVUE iff $T$ is complete. 
\end{corollary}

This follows immediately from Theorem~\ref{prop:iff}. 

In the case when a complete sufficient statistic exists, the UMVUEs can be easily characterized: 

\begin{proposition}\label{prop:compl,suff}
Suppose that a statistic $S$ is sufficient and complete. Then a statistic $T$ is a UMVUE iff for some function $u\colon\R\to\R$ one has $T=u\circ S$ except on a null set. 
\end{proposition}

\begin{proof}[Proof of Proposition~\ref{prop:compl,suff}] Given the Lehmann--Scheff\'e and Rao--Blackwell theorems, this proof is easy and presented here for readers' convenience. Indeed, the ``if'' part of Proposition~\ref{prop:compl,suff} is the Lehmann--Scheff\'e theorem itself. 

To prove the ``only if'' part, take any UMVUE $T$. Then the statistics $T_S:=\E_\th(T|S)$ and $U:=(T+T_S)/2$ are also UMVUE, and $\E_\th T_S=\E_\th U=\E_\th T$ for all $\th\in\Th$. However, by the Cauchy--Schwarz inequality, $\Var_\th U=\big(\Var_\th T_S+\Var_\th T+2\Cov_\th(T_S,T)\big)/4<\Var_\th T_S=\Var_\th T$ for some $\th\in\Th$ unless $T_S=T$ (except on a null set). It remains to note that for some function $u\colon\R\to\R$ one has $T_S=\E_\th(T|S)=u\circ S$ except on a null set. 
\end{proof}

Let us now present two more examples illustrating Theorem~\ref{prop:iff}. 

\begin{example}\label{ex:bern} (Bernoulli trials)  
Suppose that $X=\{0,1\}^n$ for some natural $n$, $\Th=(0,1)$, and $T(x):=\sum_{i=1}^n x_i$ and $\P_\th(\{x\})=\th^{T(x)}(1-\th)^{n-T(x)}$ for $x=(x_1,\dots,x_n)\in X$. 
Clearly, here for each $t\in T(X)$ the basis $B_t$ is a singleton set.
Also, the statistic $T$ here is sufficient and complete. So, it follows immediately either from Theorem~\ref{prop:iff} (cf.\ Corollaries~\ref{cor:u(T)} and \ref{cor:UMVUE,compl}) or from Proposition~\ref{prop:compl,suff} that a statistic $S$ is a UMVUE here iff for some function $u\colon\R\to\R$ one has $S=u\circ T$.   
\end{example}

\begin{example}\label{ex:beta-bern} (Beta-Bernoulli trials) Consider the following hierarchical model of $n$ independent trials, where the success probability in each trial is a random number $p$ sampled from a Beta distribution  ($p$ is sampled just once, before the trials begin); this is commonly used to model over-dispersion. Fix any positive real number $c$, which can be thought of as somewhat large.  
Suppose that $X=\{0,1\}^n$ for some natural $n$, $\Th=(-c,c)$, and $T(x):=\sum_{i=1}^n x_i$ and 
\begin{equation*}
	\begin{aligned}
	\P_\th(\{x\})=&\int_0^1 p^{T(x)}(1-p)^{n-T(x)}\,f_{c+\th,c-\th}(p)\,dp \\ 
=&\frac{\Ga(2c)}{\Ga(2c+n)}\,\frac{\Ga(c+\th+T(x))}{\Ga(c+\th)}\,\frac{\Ga(c-\th+n-T(x))}{\Ga(c-\th)}
	\end{aligned}
\end{equation*}
for $x=(x_1,\dots,x_n)\in X$, where $f_{\al,\be}$ is the probability density function of the Beta distribution with positive real parameters $\al$ and $\be$, so that $$f_{\al,\be}(p)=\frac{\Ga(\al+\be)}{\Ga(\al)\Ga(\be)}\,p^{\al-1}(1-p)^{\be-1}$$
for $p\in(0,1)$.  
One may note that then, with $(\al,\be)=(c+\th,c-\th)$, (i) $\E_\th\frac Tn=\frac\al{\al+\be}=\frac{c+\th}{2c}$ increases from $0$ to $1$ as $\th$ increases from $-c$ to $c$ and (ii) the over-dispersion $\Var_\th\frac Tn-\Var_\th\E_\th(\frac Tn|p)=\frac{n-1}{n}\,\frac{\al\be}{(\al+\be)^2(\al+\be+1)}=\frac{n-1}{n}\,\frac{c^2-\th^2}{4c^2(2c+1)}<\frac1{4(2c+1)}$ of $\frac Tn$ is small (uniformly in $\th$ and $n$) when $c$ is large. 
Using a reasoning similar to that in Example~\ref{ex:bern}, one comes to the same conclusion as there, that a statistic $S$ is a UMVUE here iff for some function $u\colon\R\to\R$ one has $S=u\circ T$. To obtain this conclusion, in this case  one only has to verify that the likelihood functions $\ell_x$ corresponding to samples $x$ with pairwise distinct values $T(x)$ are linearly independent. But this follows (by the strict total positivity property of the function $\R^2\ni(x,y)\mapsto e^{x y}$ and the P\'olya--Szeg\"o extension of the Cauchy--Binet formula for determinants) from the representation $\ell_x(\th)=\P_\th(\{x\})=K\big(T(x),\th\big)$ for all $\th\in\Th$ and $x\in X$, where  
\begin{equation*}
	K\big(t,\th\big):=
	\int_0^1 \exp\Big(t\,\ln\frac p{1-p}\Big)\exp\Big(\th\,\ln\frac p{1-p}\Big)\,\mu(d p)  
\end{equation*}
and $\mu(d p):=p^{c-1}(1-p)^{n+c-1}\,dp$;  
cf.\ e.g.\ the paragraph containing formula~(3.6) on page~12 in \cite{karlin-studden} or pages~16--17 in \cite{karlin_total-pos}.   
\end{example}


Corollaries~\ref{cor:u(T)} and \ref{cor:complete} follow as well from the mentioned Theorem~5 of Bahadur~\cite{bahadur57}, which also states that the mentioned sigma-algebra $\Si_0$ is 
complete. 

The method by which Theorem~\ref{prop:iff} and Corollaries~\ref{cor:u(T)} and \ref{cor:complete} were obtained appears to be very different from that in \ref{cor:complete}, where 
the main, and ingenious, idea was to apply implication \eqref{eq:implies} repeatedly and then use interpolation/approximation properties of polynomials. Let us present this idea here to provide the following. 

\begin{proof}[Alternative proof of Corollary~\ref{cor:complete}]
Since $T$ is a UMVUE, 
applying implication 
\eqref{eq:implies} repeatedly yields $T^2H=T(TH)\in\EE_0$, $T^3H=T(T^2H)\in\EE_0$, and so on, for any $H\in\EE_0$. Since $\EE_0$ is a linear space, it follows that 
$(f\circ T)H\in\EE_0$ for any polynomial $f$ over $\R$ and any $H\in\EE_0$. Moreover, since the set $X$ is finite, for any function $u\colon X\to\R$ there is a polynomial $f$ over $\R$ such that $u\circ T=f\circ T$. So, 
\begin{equation}\label{eq:*}
\text{$(u\circ T)H\in\EE_0$ for any function $u\colon X\to\R$ and any $H\in\EE_0$. }	
\end{equation}
In particular, if $u\circ T\in\EE_0$, then $(u\circ T)^2\in\EE_0$, that is, $\E_\th(u\circ T)^2=0$ for all $\th\in\Th$, whence $u\circ T=0$ except on a null set. Thus, $T$ is complete. 
\end{proof}

It is not clear to me if the method of \cite{bahadur57} can be used to obtain Theorem~\ref{prop:iff} of the present note. 

\bigskip

The notion of the UMVUE, which is optimal with respect to a quadratic loss function, was extended to more general loss functions $\l\colon\Th\times\R\to\R$. A loss function $\l$ is called convex if the function $\R\ni t\mapsto\l(\th,t)\in\R$ is convex for each $\th\in\Th$. 
A statistic $T$ is called a uniformly best unbiased estimator 
with respect to a loss function $\l$ ($\l$-UBUE) 
if for any statistic $S$ such that $S-T\in\EE_0$ one has 
$\E_\th\l(\th,T)\le\E_\th\l(\th,S)$ for all $\th\in\Th$. 
Obviously, if $\l(\th,t)=(t-b(\th))^2$ for some function $b\colon\Th\to\R$ and 
all $\th\in\Th$ and $t\in\R$, then any UMVUE of $b$ is an $\l$-UBUE. 
A statistic $T$ is called universally uniformly best unbiased estimator (UUBUE) if it is $\l$-UBUE  
for all convex loss functions $\l$. 

If the consideration is reduced only to statistics $T$ with a given expectation function $b$, then of interest may be the set, say $\C$, of all loss functions of the form $\l_\ll$, where $\ll$ is any differentiable strictly convex function from $\R$ to $\R$ and 
$\l_\ll(\th,t)=\ll(t)$ for all $(\th,t)\in\Th\times\R$, so that $\l_\ll(\th,t)$ does not depend on $\th$.   
The set $\C$ of loss functions was considered in \cite{schmett-strasser}. 
It is easy to see that a statistic $T$ is a UMVUE iff $T$ is an $\l_{\sq}$-UBUE, where $\sq$ is the square function given by the formula $\sq(t)=t^2$ for $t\in\R$; also, clearly $\l_{\sq}\in\C$. 

\begin{proposition}\label{prop:univ}
Take any statistic $T$ and any loss function $\l\in\C$. Then 
$T$ is a UMVUE iff $T$ is an $\l$-UBUE iff $T$ is UUBUE. 
\end{proposition}

Proposition~\ref{prop:univ} is known (see e.g.\ \cite{klebanov74} and \cite{schmett-strasser}) and is based mainly on the presented above argument by Bahadur. For readers' convenience, here is 

\begin{proof}[Proof of Proposition~\ref{prop:univ}] 
Take any UMVUE $T$. 
By \eqref{eq:*}, for any $H\in\EE_0$ one has $\E_\th(H|T)=0$ except on a null set. Take any statistic $S$ such that $H:=S-T\in\EE_0$. Then $\E_\th(S|T)=T$ except on a null set. So, for any convex loss function $\l$ and for all $\th\in\Th$ one has 
$\E_\th\l(\th,T)=\E_\th\l\big(\th,\E_\th(S|T)\big)\le\E_\th\l(\th,S)$, by Jensen's inequality. 
Thus, any UMVUE $T$ is a UUBUE. 

That any $\l$-UBUE (for any $\l=\l_\ll\in\C$) is a UUBUE is proved similarly. Here (cf.\ \cite[Proof of Satz~2]{schmett-strasser}) the most significant difference is that instead of \eqref{eq:implies} one repeatedly uses the implication 
$H\in\EE_0\implies(\ll'\circ T)H\in\EE_0$, together with the fact that the derivative $\ll'$ of the differentiable strictly convex function $\ll$ is a strictly increasing, and hence injective, function. 

That any UUBUE is an $\l$-UBUE (for any $\l\in\C$) and, in particular, is a UMVUE is trivial.  
\end{proof}

\begin{remark}
In view of Proposition~\ref{prop:univ}, one can replace the term UMVUE in Theorem~\ref{prop:iff} and Corollaries~\ref{cor:u(T)} and \ref{cor:complete} either by UUBUE or by $\l$-UBUE for any given $\l\in\C$. 
\end{remark}

\begin{proof}[Proof of Theorem~\ref{prop:iff}] 
Let us begin with some 
cleaning of the set 
$\Th$. 
Without loss of generality (w.l.o.g.), the parameter space $\Th$ is finite and the family $\PP=(\P_\th)_{\th\in\Th}$ of probability measures is linearly independent. Indeed, otherwise one can replace the family $\PP$ by any linear basis $(\P_\th)_{\th\in\Th_0}$ of $\PP$, for some $\Th_0\subseteq\Th$. Then $\Th_0$ will be nonempty and finite, since the set $\Th$ is nonempty and the set $X$ is finite. It is not hard to see that this replacement of $(\P_\th)_{\th\in\Th}$ by $(\P_\th)_{\th\in\Th_0}$ will not affect either the UMVUE property or the linear independence. 

Let us now proceed to the proof of the ``only if'' part of Theorem~\ref{prop:iff}. 
The crucial observation here is that, if a UMVUE exists, then, 
after the mentioned cleaning of the parameter space $\Th$, 
the likelihood functions $\ell_x$ corresponding to the possible samples are eigenvectors of an ``artificial'' matrix of Lagrange multipliers, and the values of the UMVUE are eigenvalues of that matrix. 

Let indeed $T$ be a UMVUE of a function $b\colon\Th\to\R$.  
For $x\in\X$ and $\th\in\Th$, introduce the abbreviations 
$$t_x:=T(x)\quad\text{and}\quad p_{\th,x}:=\P_\th(\{x\}).$$
Then
\begin{equation}\label{eq:E}
	\text{$\E_\th T^j=\sum_{x\in X}t_x^j\, p_{\th,x}$ for all $\th\in\Th$ and $j=1,2$. }
\end{equation} 

Fix for a moment any $\th\in\Th$. Then, because $T$ is a UMVUE of $b$ and in view of \eqref{eq:E}, the family $(t_x)_{x\in X}$ of the values of $T$ on $ X$ is a minimizer of 
$$
\sum_{x\in X}\tfrac12\,\tilde t_x^{\,2}\,p_{\th,x}
$$ 
over all families $(\tilde t_x)_{x\in X}$ 
in $\R$ such that 
\begin{equation*}	
\text{$\sum_{x\in X}
\tilde t_x\,p_{\tau,x}=b(\tau)$ for all $\tau\in\Th$. }
\end{equation*}
Therefore and because the family $(p_{\tau,\cdot})_{\tau\in\Th}$ is linearly independent, by the 
Euler--Lagrange multiplier rule (see e.g.\ \cite[page~441]{pourciau}), there exist Lagrange multipliers $\la_{\th,\tau}\in\R$ ($\tau\in\Th$) such that 
\begin{equation}\label{eq:}
	t_x\, p_{\th,x}=\sum_{\tau\in\Th}\la_{\th,\tau}p_{\tau,x},                  
\end{equation}
for all $x\in X$.  
Now unfix $\th\in\Th$. 

Then the system of equations \eqref{eq:} can be rewritten in matrix form: 
\begin{equation}\label{eq:eigen}
	\La \ell_x=t_x\, \ell_x\quad\text{for all }x\in X, 
\end{equation}
where $\La:=[\la_{\th,\tau}\colon\th\in\Th,\tau\in\Th]$. 
Recall that, for each $x\in X$, $\ell_x$ is the corresponding likelihood function, mapping the finite (after the cleaning) set $\Th$ to $\R$, and $\ell_x$ is identified with the corresponding column.   
Thus, \eqref{eq:eigen} means precisely that, for each $x\in X$ with the corresponding nonzero likelihood function $\ell_x$, (i) the column $\ell_x$ is an eigenvector of the square matrix $\La$ of Lagrange multipliers and (ii)  the value $t_x=T(x)$ of the UMVUE $T$ on $x$ is the corresponding eigenvalue of $\La$; this is the mentioned key observation in the proof. 
Now, to complete the proof of the ``only if'' part of Theorem~\ref{prop:iff}, it remains to recall that any family of eigenvectors of a matrix corresponding to pairwise distinct eigenvalues is linearly independent. 

Let us now turn to the proof of the ``if'' part. 
Accordingly, suppose that the  
the union $\bigcup_{t\in T(X)}B_t$ of the bases is linearly independent. 
For each $t\in T(X)$, let $\LL_t$ denote the linear span of the basis $B_t$, so that $V_t$ is a linear subspace of 
the linear space $\R^\Th$ of all functions from $\Th$ into $\R$. 
Let $U:=\sum_{t\in T(X)}\LL_t$, so that $U$ is a linear subspace of $\R^\Th$. Let $W$ be any linear subspace of $\R^\Th$ that complements $U$ to $\R^\Th$; that is, $W$ is such that $U+W=\R^\Th$ and $U\cap W=\{0\}$. Then 
each vector $v\in\R^\Th$ can be uniquely represented in the form $w+\sum_{t\in T(X)}v_t$, where $w\in W$ and $v_t\in\LL_t$ for all $t\in T(X)$. Thus, one  
has a valid definition of a linear operator $M\colon\R^\Th\to\R^\Th$ by the formula 
\begin{equation}
	M\Big(w+\sum_{t\in T(X)}v_t\Big):=\sum_{t\in T(X)}t v_t.  
\end{equation}  
Then, in particular, $M\ell_x=t_x\, \ell_x$ for all $x\in X$, where $t_x:=T(x)$, as before. 
So, letting $\La$ be the matrix of the linear operator $M$,  
one has \eqref{eq:eigen}. 
By the last sentence of the Corollary on page~440 of \cite{pourciau} to the ``Convex Multiplier Rule'', 
it now follows that $T$ is a UMVUE. 
This completes the proof of the ``if'' part and hence that of the entire theorem. 
\end{proof}

\textbf{Acknowledgment. } I am pleased to thank Lutz Mattner for pointing out that Corollary~\ref{cor:complete} is known and drawing my attention to paper \cite{kagan-malin-mattner} and references therein.

\bibliographystyle{abbrv}
\bibliography{C:/Users/ipinelis/Dropbox/mtu/bib_files/citations12.13.12}

\def\cprime{$'$} \def\polhk#1{\setbox0=\hbox{#1}{\ooalign{\hidewidth
  \lower1.5ex\hbox{`}\hidewidth\crcr\unhbox0}}}
  \def\polhk#1{\setbox0=\hbox{#1}{\ooalign{\hidewidth
  \lower1.5ex\hbox{`}\hidewidth\crcr\unhbox0}}}
  \def\polhk#1{\setbox0=\hbox{#1}{\ooalign{\hidewidth
  \lower1.5ex\hbox{`}\hidewidth\crcr\unhbox0}}} \def\cprime{$'$}
  \def\polhk#1{\setbox0=\hbox{#1}{\ooalign{\hidewidth
  \lower1.5ex\hbox{`}\hidewidth\crcr\unhbox0}}}
  \def\polhk#1{\setbox0=\hbox{#1}{\ooalign{\hidewidth
  \lower1.5ex\hbox{`}\hidewidth\crcr\unhbox0}}} \def\cprime{$'$}
  \def\cprime{$'$}
\begin{thebibliography}{1}

\bibitem{bahadur57}
R.~R. Bahadur.
\newblock On unbiased estimates of uniformly minimum variance.
\newblock {\em Sankhy\=a}, 18:211--224, 1957.

\bibitem{kagan-malin-mattner}
A.~M. Kagan, Y.~Malinovsky, and L.~Mattner.
\newblock Partially complete sufficient statistics are jointly complete.
\newblock \url{http://arxiv.org/abs/1307.3654v2}, 2014.

\bibitem{karlin_total-pos}
S.~Karlin.
\newblock {\em Total positivity. {V}ol. {I}}.
\newblock Stanford University Press, Stanford, Calif, 1968.

\bibitem{karlin-studden}
S.~Karlin and W.~J. Studden.
\newblock {\em Tchebycheff systems: {W}ith applications in analysis and
  statistics}.
\newblock Pure and Applied Mathematics, Vol. XV. Interscience Publishers John
  Wiley \& Sons, New York-London-Sydney, 1966.

\bibitem{klebanov74}
L.~B. Klebanov.
\newblock Unbiased estimators and convex loss functions.
\newblock {\em Zap. Nau\v cn. Sem. Leningrad. Otdel. Mat. Inst. Steklov.
  (LOMI)}, 43:40--52, 169, 1974.
\newblock Statistical theory of estimation, I.

\bibitem{lehmann-scheffeI}
E.~L. Lehmann and H.~Scheff{\'e}.
\newblock Completeness, similar regions, and unbiased estimation. {I}.
\newblock {\em Sankhy\=a}, 10:305--340, 1950.

\bibitem{pourciau}
B.~H. Pourciau.
\newblock Modern multiplier rules.
\newblock {\em Amer. Math. Monthly}, 87(6):433--452, 1980.

\bibitem{schmett-strasser}
L.~Schmetterer and H.~Strasser.
\newblock Zur {T}heorie der erwartungstreuen {S}ch\"atzungen.
\newblock {\em Anz. \"Osterreich. Akad. Wiss. Math.-Naturwiss. Kl.},
  (6):59--66, 1974.

\end{thebibliography}

\end{document}